\documentclass[12pt]{amsart}
\usepackage{amssymb,mathrsfs, stmaryrd}%,mathcomp,amsfonts, array}

\makeatletter \@addtoreset{equation}{section} \makeatother

\newcommand{\down}[1]{\left\lfloor #1\right\rfloor}
\newcommand{\muu}{{\boldsymbol{\mu}}}
\newcommand{\alphaa}{{\boldsymbol{\alpha}}}

\newcommand{\CC}{{\mathbb C}}
\newcommand{\QQ}{{\mathbb Q}}
\newcommand{\OOO}{{\mathscr{O}}} 

\newcommand{\ord}{\operatorname{ord}}

\newcommand{\ct}{\operatorname{ct}}

\newcommand{\Tc}{{\mathcal T}^{\mathrm{can}}}

\newcommand{\wt}{\operatorname{wt}}

\newtheorem{theorem}[equation]{Theorem}

\newtheorem{lemma}[equation]{Lemma}

\newtheorem{conjecture}[equation]{Conjecture}

\theoremstyle{definition}
\newtheorem{definition}[equation]{Definition}

\newtheorem{example}[equation]{Example}
\newtheorem{case}[equation]{}

\author{Yuri Prokhorov}
\thanks{
The author was
partially supported by 
RFBR, Nos. \ 08-01-00395-a and 06-01-72017-MHTI-a.
}
\address{
Yuri Prokhorov, Department of Higher Algebra, 
Faculty of Mathematics and Mechanics, 
Moscow State Lomonosov University, Vorobievy Gory, Moscow, 
119 899, RUSSIA}
\email{prokhoro@mech.math.msu.su} 
\title{Gap conjecture for $3$-dimensional canonical 
thresholds}
\date{}

\begin{document}
\maketitle
\begin{abstract}
We prove that the interval $(5/6,\, 1)$ contains no $3$-dimensional canonical 
thresholds.
\end{abstract}

\section{Introduction}

Let $(X\ni P)$ be a three-dimensional 
canonical singularity 
and let $S\subset X$ be a $\QQ$-Cartier divisor.
The \emph{canonical threshold} of the pair $(X,S)$ is 
\[
\ct(X,S):=\sup \{ c \mid \text{the pair $(X,cS)$ is canonical}\}.
\]
It is easy to see that $\ct(X,S)$ is rational and non-negative. 
Moreover, if $S$ is effective and integral, then $\ct(X,S)\in [0,\, 1]$.
Define the subset $\Tc_n\subset [0,\, 1]$ as follows
\[
\Tc_n:=\{\ct(X,S) \mid \text{$\dim X=n$, $S$ is 
integral and effective} 
\}.
\]

The following conjecture is an analog of 
corresponding conjectures for log canonical 
thresholds and minimal discrepancies, see \cite{Shokurov1988p},
\cite{Utah}, \cite{Kollar-1995-pairs}, \cite{Prokhorov-McKernan-2004}, \cite{kollar-2008}.

\begin{conjecture}
\label{conj-acc}
The set $\Tc_n$ satisfies the ascending chain condition.
\end{conjecture}
The conjecture is interesting for applications to 
birational geometry, see, e.g., \cite{Corti1995}.
It was shown in \cite{Birkar-Shokurov} that much more general form of \ref{conj-acc}
follows from ACC for minimal log discrepancies 
and weak Borisov-Alexeev conjecture.
The important particular case of \ref{conj-acc} is the following
\begin{conjecture}[cf. {\cite{kollar-2008}}]
\label{conj-gap}
$\epsilon_n^{\mathrm{can}}:=1-\sup (\Tc_n\setminus \{1\})>0$.
\end{conjecture}
\noindent
The aim of this note is to prove Conjecture \ref{conj-gap} for $n=3$ in a 
precise form:

\begin{theorem}
\label{main-1}
$\epsilon^{\mathrm{can}}_3=1/6$.
\end{theorem}
An analog of this theorem for log canonical thresholds was
proved by J. Koll\'ar \cite{Kollar1994}: $\epsilon^{\mathrm{lc}}_3=1/42$.

Note that replacing $(X\ni P)$ with its terminal $\QQ$-factorial
modification we may assume that $(X\ni P)$ is terminal.
Thus the following is a stronger form of Theorem \ref{main-1}:

\begin{theorem}
\label{main-2}
Let $(X\ni P)$ be a three-dimensional terminal singularity 
and let $S\subset X$ be an \textup(integral\textup) 
effective Weil $\QQ$-Cartier divisor
such that the pair $(X,S)$ is not canonical.
Then $\ct(X,S)\le 5/6$. 
Moreover, if $(X\ni P)$ is singular, then $\ct(X,S)\le 4/5$. 
\end{theorem}
The proof is rather standard.
We use the classification of terminal singularities and 
weighted blowups techniques, cf. \cite{Kawamata-1992-discr}, \cite{Kollar1994},
\cite{Markushevich-1996-discr}.

\section{Preliminaries}

\begin{case} \textbf{Notation.}
For a polynomial $\phi$, 
$\ord_0 \phi$ denotes the order of vanishing of 
$\phi$ at $0$ and
$\phi_d$ is the homogeneous 
component of degree $d$.

Throughout this paper we let $(X\ni P)$ be the germ of 
a three-dimensional terminal singularity and let 
$S\subset X$ be an 
effective Weil $\QQ$-Cartier divisor
such that the pair $(X,S)$ is not canonical.
Put $c:=\ct(X,S) >0$. Since $(X,S)$ is not canonical, $c<1$.

We work over the complex number field $\CC$.
\end{case}

\begin{lemma}
In the above notation the singularity $(S\ni P)$ is not Du Val.
\end{lemma}

\begin{proof}
This is well-known, see e.g. \cite[Th. 2.6]{Reid-1980can}.
\end{proof}

\begin{case}
We use the techniques of weighted blowups.
For definitions and basic properties we refer, for example, 
to \cite{Markushevich-1996-discr}, \cite{Reid-YPG1987}.
By fixing coordinates $x_1,\dots,x_n$
we regard the affine space $\CC^n$ as a toric variety.
Let $\alphaa=(\alpha_1,\dots,\alpha_n)$ be a weight
(a primitive lattice vector in the positive octant)
and let $\sigma_{\alpha}\colon \CC^n_{\alphaa}\to \CC^n$
be the weighted blowup with weight $\alphaa$ ($\alphaa$-blowup).
The exceptional divisor $E_{\alphaa}$ is irreducible and 
determines a discrete valuation $v_{\alphaa}$ of the function field 
$\CC(\CC^n)$ such that $v_{\alphaa}(x_i)=\alpha_i$.
\end{case}

\begin{case}
Now let $X\subset \CC^n$ be a hypersurface given by 
the equation $\phi=0$ and let $X_{\alphaa}\subset \CC^n_{\alphaa}$
be its proper transform.
Fix an irreducible component $G$ of $E_{\alphaa}\cap X_{\alphaa}$ 
such that $X_{\alphaa}$ is smooth at the generic point of $G$.
Let $v_G$ be the corresponding discrete valuation of $\CC(X)$.
Write 
\[
E_{\alphaa}\mid_{ X_{\alphaa}}=m_G G+(\text{other components}).
\]
Assume that $d_G=1$ and $G$ is not a toric subvariety 
in $\CC^n_{\alphaa}$. Then the discrepancy of $G$ with respect to $K_X$ 
is computed by the formula
\[
a(G,K_X)=|\alphaa|-1-v_{\alphaa}(\phi), \quad |\alphaa|=\sum \alpha_i,
\]
see \cite{Markushevich-1996-discr}. Let $S\subset X$ be a Cartier divisor and let $\psi$ 
be a local defining equation of $S$ in $\OOO_{0,X}$. 
Then $v_G(\psi)=v_{\alphaa}(\psi)$ and
the discrepancy of $G$ with respect to $K_X+cS$ 
is computed by the formula
\[
a(G,K_X+cS)=a(G,K_X)-cv_{G}(\psi)=|\alphaa|-1-v_{\alphaa}(\phi)-cv_{\alphaa}(\psi).
\]
Therefore,
\[
c \le a(G,K_X)/v_{\alphaa}(\psi)= (|\alphaa|-1-v_{\alphaa}(\phi))/v_{\alphaa}(\psi).
\]
\end{case}

\begin{definition}[cf. {\cite{Markushevich-1996-discr}}]
A weight $\alphaa$ is said to be \textit{admissible} if
$E_{\alphaa}\cap X_{\alphaa}$ contains at least one 
reduced non-toric component.
\end{definition}

\section{Gorenstein case}
In this section we consider the case where $(X\ni P)$ is either smooth or 
an index one singularity.
\begin{lemma}
\label{l-G-1}
If $(X\ni P)$ is smooth, then $c\le 5/6$.
\end{lemma}
\begin{proof}
Let $c>5/6$.
We may assume that $X=\CC^3$.
Let $\psi(x,y,z)=0$ be an equation of $S$.
Consider a weighted blowup 
$\sigma_{\alphaa}\colon \CC^3_{\alphaa}\to \CC^3$ with a suitable 
weight $\alphaa$. Let $E_\alphaa$ be the exceptional divisor.
Recall that $(S\ni P)$ is not Du Val.
Up to analytic coordinate change there are the following cases
(cf. \cite[4.25]{Kollar-Mori-19988}):

\begin{case} 
\textbf{Case $\ord_0 \psi\ge 3$.}
Take $\alphaa=(1,1,1)$ (usual blowup of $0$).
Then $a(E_{\alphaa}, K_X) =2$, $v_{\alphaa}(\psi)=\ord_0 \psi\ge 3$. Hence 
$c\le a(E_{\alphaa}, K_X)/v_{\alphaa}(\psi) \le 2/3$, a contradiction.
\end{case}

\begin{case} \textbf{Case $\psi=x^2+\eta(y,z)$, where $\ord_0 \eta\ge 4$.}
Take $\alphaa=(2,1,1)$.
Then $a(E_{\alphaa}, K_X) =3$, $v_{\alphaa}(\psi)=4$. Hence 
$c\le a(E_{\alphaa}, K_X) /v_{\alphaa}(\psi)\le 3/4$, a contradiction.
\end{case}

\begin{case} \textbf{Case $\psi=x^2+y^3+\eta(y,z)$, where $\ord_0 \eta\ge 4$.}
Here $\eta$ contains no terms $yz^l$, $l\le 3$ and $z^l$, $l\le 5$
(see, e.g., \cite[4.25]{Kollar-Mori-19988}).
Take $\alphaa=(3,2,1)$.
Then $a(E_{\alphaa}, K_X) =5$, $v_{\alphaa}(\psi)=6$. Hence 
$c\le a(E_{\alphaa}, K_X) /v_{\alphaa}(\psi)= 5/6$, a contradiction.
\end{case}
\end{proof}

\begin{lemma}
\label{l-G-2}
Assume that $(X\ni P)$ is a Gorenstein terminal singularity
and $(X\ni P)$ is not smooth. Then $c\le 4/5$.
\end{lemma}

\begin{proof}
Let $c>4/5$. 
We may assume that $X$ is a hypersurface in $\CC^4$
(it is an isolated cDV-singularity \cite{Reid-1980can}).
Let $\phi(x,y,z,t)=0$ be the equation of $X$.
Since $(X\ni P)$ is a cDV-singularity, 
$\ord_0 \phi=2$. According to \cite{Markushevich-1996-discr},
in a suitable coordinate system $(x,y,z,t)$, there is an admissible weighted blowup
$\sigma_{\alphaa} \colon \CC^4_{\alphaa}\to \CC^4$ such that at least 
for one component $G$ of $E_{\alphaa}\cap X_{\alphaa}$ 
we have $a(G,K_X)=1$. Then $c\le 1/v_{\alphaa}(\psi)$, so $v_{\alphaa}(\psi)=1$.
This means, in particular, that $\ord_0 \psi=1$. 
Up to coordinate change we may assume that
$\psi=t$. Write
\[
\phi=\eta(x,y,z)+t \zeta (x,y,z,t).
\]
Then $S$ is a hypersurface in $\CC^3_{x,y,z}$
given by $\eta(x,y,z)=0$.
As in the proof of Lemma \ref{l-G-1}, 
using Morse Lemma we get the following cases:

\begin{case} \textbf{Case $\ord_0 \eta\ge 3$.}
Take $\alphaa=(1,1,1,2)$. 
By the terminality condition \cite[Th. 4.6]{Reid-YPG1987}, 
we have $4=v_{\alphaa}(xyzt)-1>v_{\alphaa}(\phi)$.
Hence, $v_{\alphaa}(\eta)\le 3$ and $\eta_3\neq 0$.
We claim that $\alphaa$ is 
admissible whenever $\eta_3$ is not a cube of a linear form.
Indeed, in the affine chart $U_x:=\{x\neq 0\}$
the map $\sigma_{\alphaa}^{-1}$ is given by 
\begin{equation}
\label{eq-odd3}
x \mapsto x',\quad y \mapsto y'x',\quad z\mapsto z'x',\quad 
t \mapsto t'x'^2.
\end{equation}
First we assume that $\zeta$ contains the term $x$.
After the coordinate change $x\mapsfrom \zeta (x,y,z,t)$ we obtain
\[
\phi=\eta(x,y,z)+tx.
\]
Using \eqref{eq-odd3} we see that
$E_{\alphaa}\cap X_{\alphaa}$ is given in $\sigma_{\alphaa}^{-1}(U_x)\simeq 
\CC^4$ by 
\[
x'=\eta_3(1,y',z')+t'=0.
\]
Hence $\alphaa$ is admissible, i.e., 
$E_{\alphaa}\cap X_{\alphaa}$ has
a reduced non-toric component $G$. Then 
$a(G, K_X) =1$, 
$v_{G}(\psi)=2$ and $c\le a(G, K_X) /v_{G}(\psi)=1/2$, a contradiction.

Thus by symmetry we may assume that $\zeta$ contains no terms $x$, $y$, $z$.
Since $\ord_0 \phi=2$, $\zeta$ contains $t$. So, 
\[
\phi=\eta(x,y,z)+t^2+ t\xi (x,y,z,t),\quad \ord_0 \xi\ge 2.
\]
As above, $E_{\alphaa}\cap X_{\alphaa}$ is given in $\CC^4$ by 
$x'=\eta_3(1,y',z')=0$.
If $\eta_3$ is not a cube of a linear 
form, then $E_{\alphaa}\cap X_{\alphaa}$ has a reduced non-toric component $G$. 
Then, as above, 
$c\le 1/2$, a contradiction.

Consider the case where $\eta_3$ is a cube of a linear form. 
Then we may assume that 
$\eta_3(x,y,z)=y^3$, so 
\[
\phi=y^3+\eta^{\bullet}(x,y,z)+t^2+ t\xi (x,y,z,t),\quad \ord_0 \xi\ge 2, \quad 
\ord_0\eta^{\bullet} \ge 4.
\]
Put $\alphaa'=(2,2,2,3)$. Again, 
in the affine chart $U_x:=\{x\neq 0\}$
the map $\sigma_{\alphaa'}^{-1}$ is given by 
$x \mapsto x'^2$, $y \mapsto y'x'^2$, $z\mapsto z'x'^2$,
$t \mapsto t'x'^3$, 
where $\sigma_{\alphaa'}^{-1}(U_x)\simeq \CC^4/\muu_2(1,0,0,1)$ and 
\[
E_{\alphaa'}\cap X_{\alphaa'}\cap \sigma_{\alphaa'}^{-1}(U_x)=
\{ 
x'=0,\ y'^3+t'^2 =0\}.
\]
Thus $\alphaa'$ is admissible and for some component $G'$ of 
$X_{\alphaa'}\cap E_{\alphaa'}$ we have
$a(G', K_X)=2$, $v_{G'}(\psi)=3$, $c\le 2/3$, a contradiction.
\end{case}

\begin{case} \textbf{Case $\eta=x^2+\xi(y,z)$, where $\ord_0 \xi\ge 4$.}
By Morse Lemma we may assume that $\zeta$ does not depend on $x$.
Write $\zeta_1=\delta_1y+\delta_2z+\delta_3t$, $\delta_i\in \CC$.
Take $\alphaa=(2,1,1,3)$. 
In the affine chart $U_y:=\{y\neq 0\}$
the map $\sigma_{\alphaa}^{-1}$ is given by 
$x \mapsto x'y'^2$, $y \mapsto y'$, $z\mapsto z'y'$,
$t \mapsto t'y'^3$ and  
\[
E_{\alphaa}\cap X_{\alphaa}\cap \sigma_{\alphaa}^{-1}(U_y)=
\{ y'=0,\ x'^2+\xi_4(1,z')+\delta_1t' +\delta_2t'z'=0\}.
\]
If either $\delta_1\neq 0$ or $\delta_2\neq 0$ or $\xi_4\neq 0$, then 
$E_{\alphaa}\cap X_{\alphaa}$ is reduced (at least over $U_y$). 
Hence, 
$\alphaa$ is 
admissible and for some component $G$ of $E_{\alphaa}\cap X_{\alphaa}$
we have
$c\le a(G, K_X) /v_{G}(\psi)= 2/3$, a contradiction.
Thus $\delta_1=\delta_2= 0$ and $\xi_4=0$. 
Then we can write
\[
\phi=x^2+\xi(y,z)+\delta_3t^2+t\zeta^{\bullet}(y,z,t), \quad \ord_0\xi\ge 5,
\quad \ord_0\zeta^{\bullet}\ge 2.
\]
Take $\alphaa'=(2,1,1,2)$. 
In the affine chart $U_y:=\{y\neq 0\}$
the map $\sigma_{\alphaa'}^{-1}$ is given by 
$x \mapsto x'y'^2$, $y \mapsto y'$, $z\mapsto z'y'$,
$t \mapsto t'y'^2$ and  
\[
E_{\alphaa'}\cap X_{\alphaa'}\cap \sigma_{\alphaa'}^{-1}(U_y)=
\{ y'=0,\ x'^2+\delta_3t'^2+t\lambda(1,z')=0\},
\]
where $\lambda$ is the degree $2$ 
homogeneous part of $\zeta(y,z,0)$.
If $\delta_3\neq 0$ or $\lambda\neq 0$, 
as above, $\alphaa'$ is 
admissible and $c\le 1/2$, a contradiction.
Thus $\delta_3=0$, $\lambda= 0$, and
\[
\phi=x^2+\xi(y,z)+\delta t^3+t\zeta^{\circ}(y,z,t), \quad \delta\in \CC, \quad\ord_0\xi\ge 5,
\quad \ord_0\zeta^{\circ}\ge 3.
\]
Applying the terminality condition \cite[Th. 4.6]{Reid-YPG1987}
with weight $(2,1,1,1)$ we get that $\delta\neq 0$.

Take $\alphaa''= (3,1,1,2)$. As above we get that $\alphaa''$ 
is admissible and then
$c\le 1/2$, a contradiction.
\end{case}

\begin{case} \textbf{Case $\eta=x^2+y^3+\xi(y,z)$, where $\ord_0 \xi\ge 4$.}
Here $\xi$ contains no terms $yz^l$, $l\le 3$ and $z^l$, $l\le 5$
(see, e.g., \cite[4.25]{Kollar-Mori-19988}).
Write $\zeta_1=cz+\ell(x,y,t)$ and 
$\xi=\xi_{(6)}+\xi_{(7)}+\cdots$, where $\xi_{(d)}$ is 
the degree $d$ 
weighted homogeneous part of $\xi$ with respect to $\wt(y,z)=(2,1)$.
Here $\xi_{(6)}$ is a linear combination of $z^6$, $yz^4$,  $y^2z^2$.
Take $\alphaa=(3,2,1,5)$.
In the affine chart $U_z:=\{z\neq 0\}$
the map $\sigma_{\alphaa}^{-1}$ is given by 
$x \mapsto x'z'^3$, $y \mapsto y'z'^2$, $z\mapsto z'$,
$t \mapsto t'z'^5$ and
\[
E_{\alphaa}\cap X_{\alphaa}\cap \sigma_{\alphaa}^{-1}(U_z)=
\{z'=0,\ x'^2+y'^3+\xi_{(6)}(y',1)+\delta t'=0\},
\]
where $\delta$ is a constant and $\xi_{(6)}(y',1)$ contains no $y'^3$.
Hence $\alphaa$ is admissible, i.e., 
$E_{\alphaa}\cap X_{\alphaa}$ has
a reduced non-toric component $G$. Then 
$a(G, K_X) =4$, $v_{G}(\psi)=5$, and
$c\le a(G, K_X) /v_{G}(\psi)\le 4/5$, a contradiction.
\end{case}
\end{proof}

The following examples show that 
bounds $\ct(X,S)\le 5/6$ and $\le 4/5$ in 
Theorem \ref{main-2} are sharp.

\begin{example}
Let $X=\CC^3$ and let $S=S^d$ is given by $x^2+y^3+z^d$, $d\ge 6$.
Then $\ct(\CC^3,S^d)=5/6$. We prove this by descending induction on $\down{d/6}$.
Take $\alphaa=(3,2,1)$ and consider the $\alphaa$-blowup 
$\sigma_{\alphaa}\colon \CC^3_{\alphaa}\to \CC^3$.
Let $S_{\alphaa}\subset X_{\alphaa}$ be the proper transform of $S$.
We have $a(E_{\alphaa}, K_X) =5$ and $v_{\alphaa}(\psi)=6$. Hence, $\ct(\CC^3,S^d)\le 5/6$.
Further,
\[
\textstyle
K_{\CC^3_{\alphaa}}+\frac56S_{\alphaa}=\sigma_{\alphaa}^*(K_{\CC^3}+\frac56S).
\]
Thus it is sufficient to show that $\ct(X_{\alphaa},\frac 56 S_{\alphaa})$ is canonical.
We have three affine charts: 
\begin{itemize}
 \item 
$U_x:=\{x\neq 0\}$. Here $\sigma_{\alphaa}^{-1}\colon$ 
$x \mapsto x'^3$, $y \mapsto y'x'^2$, $z\mapsto z'x'$,
$S_{\alphaa}$ is given in $\sigma_{\alphaa}^{-1}(U_x)\simeq\CC^3/\muu_3(-1,2,1)$ by the equation 
$1+y'^3+z'^dx'^{d-6}=0$. Hence, in this chart, 
$S_{\alphaa}$ is smooth and does not pass through a (unique) singular point
of $\sigma_{\alphaa}^{-1}(U_x)$.
 \item 
$U_y:=\{y\neq 0\}$. Here $\sigma_{\alphaa}^{-1}\colon$ 
$x \mapsto x'y'^3$, $y \mapsto y'^2$, $z\mapsto z'y'$,
$S_{\alphaa}$ is given in $\sigma_{\alphaa}^{-1}(U_y)\simeq \CC^3/\muu_2(3,-1,1)$ by the equation 
$x'^2+1+z'^dy'^{d-6}=0$. Again, in this chart, 
$S_{\alphaa}$ is smooth and does not pass through a (unique) singular point
of $\sigma_{\alphaa}^{-1}(U_y)$.

 \item 
$U_z:=\{z\neq 0\}$. Here $\sigma_{\alphaa}^{-1}\colon$ 
$x \mapsto x'z'^3$, $y \mapsto y'z'^2$, $z\mapsto z'$,
$S_{\alphaa}$ is given in $\sigma_{\alphaa}^{-1}(U_z)\simeq \CC^3$ by the equation 
$x'^2+y'^3+z'^{d-6}=0$. In this chart, $(X_{\alphaa},S_{\alphaa})\simeq (\CC^3, S^{d-6})$.
\end{itemize}
Thus $X_{\alphaa}$ has only terminal singularities,
$S_{\alphaa}$ does not pass through any singular point of 
$X_{\alphaa}$, 
and the pair $(X_{\alphaa}, S_{\alphaa})$ is terminal in
charts $U_x$ and $U_y$. In the chart $U_z$ the pair
by induction $(X_{\alphaa},\frac 56 S_{\alphaa})$ is canonical
(moreover, $(X_{\alphaa}, S_{\alphaa})$ is canonical if $d\le 11$).
Therefore, $\ct(X,S)=5/6$.
\end{example}

\begin{example}
Let $X\subset \CC^4$ is given by 
$x^2+y^3+z^d+tz=0$, $d\ge 7$ and let $S$ cut out by $t=0$.
Take $\alphaa=(3,2,1,5)$ and
consider the $\alphaa$-blowup $\sigma_{\alphaa} \colon X_{\alphaa}\to X$.
Let $S_{\alphaa}\subset X_{\alphaa}$ be the proper transform of $S$.
We see below that $\alphaa$ is admissible.
Moreover, the exceptional divisor $G:=E_{\alphaa}\cap X_{\alphaa}$
is reduced and irreducible.
We have four charts: 
\begin{itemize}
 \item 
$U_x:=\{x\neq 0\}$. Here $\sigma_{\alphaa}^{-1}\colon
x \mapsto x^3$, $y \mapsto yx^2$, $z\mapsto zx$, $t\mapsto tx^5$,
$X_{\alphaa}$ is given in $\sigma_{\alphaa}^{-1}(U_x)\simeq\CC^4/\muu_3(-1,2,1,5)$ by the equation 
$1+y^3+z^dx^{d-6}+tz=0$ and $S_{\alphaa}$ by two equations $x=1+y^3+tz=0$. 
Hence, in this chart, both $X_{\alphaa}$ and
$S_{\alphaa}$ are smooth. 
 \item 
$U_y:=\{y\neq 0\}$. Here
$\sigma_{\alphaa}^{-1}\colon 
x \mapsto xy^3$, $y \mapsto y^2$, $z\mapsto zy$, $t\mapsto ty^5$,
$\sigma_{\alphaa}^{-1}(U_y)\simeq \CC^4/\muu_2(3,-1,1,5)$, 
$X_{\alphaa}=\{x^2+1+z^dy^{d-6}+tz=0\}$, and $S_{\alphaa}=\{y=x^2+1+tz=0\}$. 
As above, both $X_{\alphaa}$ and
$S_{\alphaa}$ are smooth in this chart.
 \item 
$U_z:=\{z\neq 0\}$. Here $\sigma_{\alphaa}^{-1}\colon
x \mapsto xz^3$, $y \mapsto yz^2$, $z\mapsto z$, $t\mapsto tz^5$,
$\sigma_{\alphaa}^{-1}(U_z)\simeq \CC^4$,   
$X_{\alphaa}=\{x^2+y^3+z^{d-6}+t=0\}$, and $S_{\alphaa}=\{z=x^2+y^3+t=0\}$. 
As above, both $X_{\alphaa}$ and
$S_{\alphaa}$ are smooth in this chart.

 \item 
$U_t:=\{t\neq 0\}$. Here $\sigma_{\alphaa}^{-1}\colon
x \mapsto xt^3$, $y \mapsto yt^2$, $z\mapsto zt$, $t\mapsto t^5$,
$\sigma_{\alphaa}^{-1}(U_t)\simeq \CC^4/\muu_5(3,2,1,-1)$,
$X_{\alphaa}=\{x^2+y^3+z^dt^{d-6}+z=0\}$, and $S_{\alphaa}=\{t=x^2+y^3+z=0\}$. 
The variety
$X_{\alphaa}$ has a unique singular point $Q$ at the origin and this point is terminal of type 
$\frac15(3,2,-1)$ the surface
$S_{\alphaa}$ is smooth and does not pass through $Q$. 
\end{itemize}
Thus we have $a(G, K_X) =4$, $v_{\alphaa}(\psi)=5$, and $a(G, K_X+ \frac45 S) =0$.
Therefore, 
\[
\textstyle
K_{X_{\alphaa}}+\frac45S_{\alphaa}=\sigma_{\alphaa}^*(K_{X}+\frac45S).
\]
Since the pair $K_{X_{\alphaa}}+\frac45S_{\alphaa}$ is canonical, $\ct(X,S)=4/5$.
\end{example}

\section{Non-Gorenstein case}
Now we assume that $(X\ni P)$ is a (terminal) point of index $r>1$.
Let $\pi \colon (X^\sharp\ni P^\sharp)\to (X\ni P)$ be the index-one
cover and let $S^\sharp:=\pi^{-1}(S)$.

\begin{lemma}
\label{l-nG-1}
If $(X\ni P)$ is a cyclic quotient singularity, then $\ct(X,S)\le 1/2$.
\end{lemma}

\begin{proof}
By our assumption we have
$X\simeq \CC^3/\muu_r(a,-a,1)$ for some $r\ge2$, $1\le a <r$, $\gcd(a,r)=1$.
Assume that $c=\ct(X,S)> 1/2$. 
Let 
$\psi=0$ be a defining equation of $S^\sharp$.
Consider the weighted blowup $\sigma_{\alphaa}\colon X_{\alphaa}\to X$
with weights $\alphaa =\frac1r (a,r-a,1)$.
Then $a(E_{\alphaa}, K_X)=1/r$.
Since $a(E_{\alphaa}, K_X)-c v_\alphaa (\psi)\ge 0$, we have 
$v_\alphaa (\psi)\le a(E_{\alphaa}, K_X)/c< 2a(E_{\alphaa}, K_X)= 2/r$
and so $v_\alphaa (\psi)=1/r$.
Thus we may assume that $\psi$ contains $x_3$
(if $a\equiv \pm 1$ we possibly have to permute coordinates).
Then $S^\sharp\simeq \CC^2$ is smooth and $S\simeq \CC^2/\muu_r(a,-a)$, i.e.,
$S$ is Du Val of type $A_{r-1}$.
\end{proof}

\begin{lemma}
\label{l-nG-2}
If $(X\ni P)$ is a terminal singularity of index $r>1$ 
and $\ct(X,S)>1/2$, then $K_X+S\sim 0$.
\end{lemma}

\begin{proof}
By Lemma \ref{l-nG-1} $(X\ni P)$ is not a cyclic quotient singularity.
There is an analytic $\muu_r$-equivariant 
embedding $(X^\sharp, P^\sharp) \subset (\CC^4,0)$. 
Let $(x_1,x_2,x_3,x_4)$ be coordinates in $\CC^4$, let $\phi=0$ be an 
equation of $X^\sharp$, and let $\psi=0$ be an equation of $S^{\sharp}$. 
We can take $(x_1,x_2,x_3,x_4)$ and $\phi$ 
to be semi-invariants 
such that one of the following holds \cite{Reid-YPG1987}:
\begin{enumerate}
\item[-] 
\textbf{Main series.}
$\wt (x_1,x_2,x_3,x_4; \phi)\equiv (a,-a, 1, 0;0)\mod r$, where $\gcd (a,r)=1$.
\item[-] 
\textbf{Case $cAx/4$.} $r=4$,
$\wt (x_1,x_2,x_3,x_4; \phi)\equiv (1,3, 1, 2;2)\mod 4$.
\end{enumerate}
In both cases $\wt (x_1x_2x_3x_4)- \wt \phi\equiv \wt x_3\mod r$.
According to \cite{Kawamata-1992-discr} there is a weight 
$\alphaa$
such that for the corresponding $\alphaa$-blowup
$\sigma_{\alphaa}\colon X_{\alphaa}\subset W\to X\subset \CC^4/\muu_r$ 
the exceptional divisor $E_{\alphaa}\cap X_{\alphaa}$ 
has a reduced component $G$ of 
discrepancy $a(G, K_X)=1/r$.
Moreover, $r\alpha_i\equiv \wt x_i \mod r$, $i=1,2,3,4$.
Since $c>1/2$, we have $1/r-cv_\alphaa (\psi)\ge 0$, i.e., 
$rv_\alphaa (\psi)< 2$, so $rv_\alphaa (\psi)=1$. 
In particular, $\wt \psi\equiv 1 \mod r$. 

Let $\omega$ 
be a section of $\OOO_X(-K_X)$. Then $\omega$ can be written as 
\[
\omega=\lambda (\partial \phi/\partial x_4)(dx_1 \wedge dx_2\wedge dx_3)^{-1},
\]
where
$\lambda$ is a semi-invariant function with 
\[
\wt \lambda-\wt (x_1x_2x_3x_4)+\wt \phi\equiv \wt\omega \equiv 0\mod r.
\]
Thus, $\wt \psi \equiv \wt \lambda\mod r$. Hence, $S\sim -K_X$.
\end{proof}

\begin{lemma}
\label{l-nG-3}
If $(X\ni P)$ is a terminal singularity of index $r>1$, 
then $c\le 4/5$.
\end{lemma}

\begin{proof}
Since $\pi$ is \'etale in codimension one, 
we have $K_{X^\sharp }+cS^\sharp=\pi^*(K_X+cS)$. Hence
the pair $(X^\sharp,\, cS^\sharp)$ is canonical (see, e.g., 
\cite[3.16.1]{Kollar-1995-pairs}). Assume that $c> 4/5$.
By Lemma \ref{l-nG-1} the point $(X^\sharp\ni P^\sharp)$ is singular.
Then by Lemma \ref{l-G-2} the pair 
$(X^\sharp,\, S^\sharp)$ is canonical.
Therefore, $(S^\sharp\ni P^\sharp)$ is a Du Val singularity.
Then the singularity $(S\ni P)=(S^\sharp\ni P^\sharp)/\muu_r$
is log terminal. On the other hand, by Lemma \ref{l-nG-2}
the divisor $K_S$ is Cartier. Hence, $(S\ni P)$ is Du Val, a contradiction.
\end{proof}

%\bibliography{my_ref}
%\bibliographystyle{alpha}%{plain}%{alpha}%%{unrst} {abbrv}

\end{document}